\documentclass[11pt]{article}
\usepackage[letterpaper, left=1.truein, right=.8truein, top = 1.truein, bottom = 1.truein]{geometry}

\usepackage{eulervm}
\usepackage{tgpagella}
\usepackage[T1]{fontenc}
\usepackage{amsmath}
\usepackage{amssymb}
\usepackage{amsthm}
\usepackage{cite}
\usepackage{color}
\usepackage{graphicx}
\usepackage{color}
\usepackage{bm}
\usepackage{mathtools}

\DeclareMathAlphabet{\mathpzc}{OT1}{pzc}{m}{it}

\DeclareMathOperator{\ucat}{ucat}

\usepackage{algorithm}
\usepackage{algpseudocode}
\makeatletter \def\BState{\State\hskip-\ALG@thistlm} \makeatother

\usepackage{etoolbox} \makeatletter
\patchcmd{\@maketitle}{\begin{center}}{\begin{flushleft}}{}{}
\patchcmd{\@maketitle}{\begin{tabular}[t]{c}}{\begin{tabular}[t]{@{}l}}{}{}
\patchcmd{\@maketitle}{\end{center}}{\end{flushleft}}{}{}

\newcommand\be{\begin{equation}}
\newcommand\ee{\end{equation}}

\newtheorem{thm}{Theorem}[section]

\newtheorem{lem}[thm]{Lemma}
\newtheorem{prp}[thm]{Proposition}
\newtheorem{cor}[thm]{Corollary}
\newtheorem{dfn}[thm]{Definition}
\newtheorem{rem}[thm]{Remark}
\newtheorem{obs}[thm]{Observation}

\newcommand{\N}{\mathbb{N}}

\newcommand{\supp}{\text{supp}}

\DeclarePairedDelimiterX\set[1]\lbrace\rbrace{#1}

\newcommand{\R}{\ensuremath{\mathbb{R}}}

\DeclareMathOperator{\cone}{cone}
\DeclareMathOperator{\chr}{chr}

\usepackage{algorithm}
\usepackage{mathtools}
\usepackage[shortlabels]{enumitem}
\usepackage{algpseudocode}
\usepackage{hyperref}
\usepackage{tikz}
\usetikzlibrary{backgrounds}
\usetikzlibrary{intersections}
\usepackage{bbm}

\usetikzlibrary{positioning}

\makeatletter \def\BState{\State\hskip-\ALG@thistlm} \makeatother

\usepackage{etoolbox} \makeatletter
\patchcmd{\@maketitle}{\begin{center}}{\begin{flushleft}}{}{}
\patchcmd{\@maketitle}{\begin{tabular}[t]{c}}{\begin{tabular}[t]{@{}l}}{}{}
\patchcmd{\@maketitle}{\end{center}}{\end{flushleft}}{}{}

\usepackage{graphicx}
\usepackage{subcaption}
\usepackage{layouts}

\begin{document}

\author{Mishal Assif P K$^1$, Yuliy Baryshnikov$^{1,2,3}$ }
\date{%
    $^1$University of Illinois, Department of ECE\\%
    $^2$University of Illinois, Department of Mathematics\\
    $^3$Kyushu university, IMI\\[2ex]%
    \today
}

\title{Minimal Unimodal Decomposition is NP-Hard on Graphs}


\maketitle

\abstract{
     A function on a topological space is called unimodal if all of its super-level sets are contractible. A minimal unimodal decomposition of a function $f$ is the smallest number of unimodal functions that sum up to $f$. The problem of decomposing a given density function into its minimal unimodal components is fundamental in topological statistics. We show that finding a minimal unimodal decomposition of an edge-linear function on a graph is NP-hard. Given any $k \geq 2$, we establish the NP-hardness of finding a unimodal decomposition consisting of $k$ unimodal functions. We also extend the NP-hardness result to related variants of the problem, including restriction to planar graphs, inapproximability results, and generalizations to higher dimensions.}

\section{Introduction}
\label{sec:intro}

A central theme in statistics and data analysis is the separation of signal from noise. Data is often modeled as noisy observations concentrated around underlying signals, idealized as points or structures in an abstract space where the measurements reside. This conception gives rise to a diffuse distribution of observed values around these signal centers. A classical and widely used model in this context is the Gaussian mixture model, where each component models a signal, and the noise is represented by the spread of the Gaussian.

While powerful, the class of Gaussian mixtures is intrinsically limited: it depends on finitely many parameters and is rigid in its geometric assumptions. To overcome these limitations, one may consider broader classes of functions such as log-concave or quasi-concave functions, where the superlevel sets $\{f \geq c\}$ are convex or empty. These generalizations preserve certain structural features that are useful in analysis and optimization. However, they remain fundamentally tied to the linear structure of the underlying space, also making their extension to manifolds or more general topological spaces challenging.

\subsection{Categories and Coordinate-Free Generalizations}

A more flexible approach to the notion of signal decomposition, especially in non-Euclidean settings, is to adopt coordinate-free analogs grounded in topology. Instead of convexity, one can consider \textit{contractibility} of superlevel sets: a very general topological property that retains the essence of being "simple" or "non-fragmented" while eschewing dependence on linear or metric structure. This leads to a conceptual shift where statistical modes are defined not through convexity or parametric form (like Gaussians), but via topological simplicity.

This shift mirrors a classical correspondence between geometry and topology. In computational geometry, for instance, one encounters the problem of covering a polygonal domain $D \subset \mathbb{R}^n$ with the smallest number of convex sets, an optimization problem whose solution offers a measure of the geometric complexity of $D$. The natural topological analog is to cover a topological space $X$ with the smallest number of contractible subsets. The minimal such number is known as the \emph{geometric category} of $X$, a homeomorphism invariant (see e.g. \cite{james_category_1978}). Several other notions of topological decomposition have emerged from this perspective. The most prominent is the \emph{Lusternik–Schnirelmann category}, which uses subsets that are null-homotopic in $X$. Other related concepts include the \emph{sectional category} (also known as the Schwarz genus) of a fibration, and the more recent notion of \emph{topological complexity}, which arises in the study of path planning in robotics (see \cite{Far03}). These invariants attempt to quantify how topologically intricate a space is by measuring how simply it can be covered by simple (contractible) parts.


This paper focuses on a weighted, statistical version of these ideas, introduced in \cite{baryshnikov_unimodal_2011}. Given a density function $f: X \to [0, \infty)$ on a topological space $X$, we ask: what is the minimal number of \emph{modes} required to write $f$ as a sum of unimodal functions, each supported on a contractible subset? This quantity serves as a topological analog of the number of components in a Gaussian mixture. 

By moving away from rigid geometric prototypes and toward a topological definition of unimodality (based on the contractibility of superlevel sets), we obtain a more adaptable framework. This allows us to define and study decompositions in spaces where traditional convexity-based or parametric approaches are not applicable. The resulting theory is robust to the underlying topology and thus better suited to modern applications in topological data analysis.

In this work, we study the computational complexity of such decompositions. Specifically, we show that the problem of computing $\ucat^p(f)$ (see definitions in the next section), the minimal number of topological modes required to represent $f$, is NP-hard in a variety of settings.

\subsection{Unimodality}
This paper concerns the coordinate-free topological version of the problem of minimal decompositions of densities over a space, a statistical analogue of the category of a domain. The resulting invariant is called the unimodal category \cite{baryshnikov_unimodal_2011}.

\begin{dfn}[Unimodal Function]
A function $f : X \rightarrow [0, \infty)$ on the topological space $X$ is unimodal if the superlevel sets $\{ f \geq c \}$ is contractible or empty for all $c > 0$.
\end{dfn}

\begin{dfn}[Unimodal Category]
For any $1 <= p < \infty$ and function $f : X \rightarrow [0, \infty)$, the unimodal category of $f$, denoted  $\ucat^p(f)$, is the smallest number $k$ such that there exist unimodal functions $f_1, f_2, \ldots, f_{k}$ with
\[
	\sum_{i=1}^{k} f_i^p(x) = f^p(x),\mbox{ for all } x \in X.
\] 
Similarly, for $p = \infty$, $\ucat^{\infty}(f)$ is the smallest number $k$ such that there is a set of unimodal functions $f_1, f_2, \ldots, f_{k}$ with
\[
	\max_{i=1}^{k} f_i(x) = f(x),\mbox{ for all }x \in X.
\]
\end{dfn}

A set of $\ucat^p(f)$ unimodal functions that sum up to $f$ is called a minimal unimodal decomposition. It should be noted that minimal unimodal decompositions need not be unique. 

\begin{dfn}[Strong Unimodal Category]
For any $1\leq p < \infty$ and function $f : X \rightarrow [0, \infty)$, the strong unimodal category of $f$ $\ucat_s^p(f)$ is the smallest number $k$ such that there exists unimodal functions $f_1, f_2, \ldots, f_{\ucat_s^p(f)}$ with
\[
	\sum_{i=1}^{k} f_i^p(x) = f^p(x),\mbox{ for all } x \in X.
\] 
and the intersection of any collection of superlevel sets of the functions $\bigcap\limits_{l=1}^{m} \{f_{i_l} > c_l\}$
is contractible or empty.
\end{dfn}
This definition aims to reflect the property of convex sets to be closed under intersections, and to ensure that the standard tools, like the {\em Nerve Lemma}, work for our setting.

\begin{rem}
For $1 < p < \infty$, $\ucat^p(f) = \ucat^1(f^p)$ and $\ucat_s^p(f) = \ucat_s^1(f^p)$.
\end{rem}

\subsection{Functions on Graphs}
We will denote by $G=(V,E)$ a graph with vertices $V$ and edges $E$, containing no self loops or multiple edges. We think of $G$ as a geometric simplicial complex, so that every point $e$ on an edge $(v_0, v_1)$  can be
endowed with a barycentric coordinate $c \in (0,1)$ such that $e = cv_0 + (1-c)v_1$. A function $f : G \rightarrow [0,\infty)$ is edge linear if 
\[
f(cv_0 + (1-c)v_1) = cf(v_0) + (1-c)f(v_1)
\]
 for all edges $(v_0,v_1) \in E$. We will deal almost exclusively with edge linear functions in this article. This does not lead to much loss in generality. If $f$ is a continuous function on the graph with finitely many critical points, we can add new vertices at each critical point of $f$ to the graph. All the critical points of $f$ are vertices in the resulting augmented graph, and so $f$ is monotonic along each edge of the new graph. If $f$ is monotonic along each edge of the graph, we can simply keep the values of $f$ at the vertices of the graph and modify the coordinates along each edge so that the function is edge linear. Then $\ucat^p(f)$ is invariant under this modification, as shown in \cite{Bar20}.

 \subsection{Outline of Results}

We list the key results of the article here. Let $G = (V, E)$ be a graph and $f$ an edge-linear function on $G$, unless otherwise mentioned.

\begin{enumerate}
    \item For any $p \in \mathbb{N} \cup \{\infty\}$ and any fixed $k >= 2$, it is NP-hard to determine whether $\ucat^p(f) \leq k$ (see Theorem \ref{th:main}).
    \item For any $p \in \mathbb{N}$, it is NP-hard to determine $\ucat^p(f)$ or $\ucat_s^p(f)$ (see Theorem \ref{th:main2}), even if the graph $G$ is restricted to being planar and having maximum degree less than 3 (see Corollary \ref{cor:cor1}). 
    \item For any $p \in \mathbb{N}$, it is NP-hard to approximate $\ucat^p(f)$ or $\ucat_s^p(f)$ by a factor strictly less than $\sqrt{2}$ (see Corollary \ref{cor:cor2}).
    \item For piecewise linear functions $f$ on two dimensional simplicial complexes  homeomorphic to the unit square in $\mathbb{R}^2$, it is NP-hard to determine  $\ucat^p(f)$ or $\ucat_s^p(f)$ for any $p \in \mathbb{N}$ (see Theorem \ref{thm:highd}).
\end{enumerate}

Essentially, when the underlying space is a 1 dimensional complex (a graph), the minimal unimodal decomposition problem can be solved in polynomial time only if the space is contractible (a tree). However, if the underlying space is a complex of dimension at least 2, the minimal unimodal decomposition problem cannot be solved in polynomial time, even if the space is contractible.

\section{Minimal Unimodal Decomposition on Trees}

The problem of determining minimal unimodal decompositions for functions on trees was solved in \cite{Bar20}. We summarize some results of \cite{Bar20} that are relevant to this article in this section. Unimodal functions on trees can be characterized by the following theorem.

\begin{thm}
\label{char-unimod-tree}
Let $G$ be a tree and $f : G \rightarrow [0,\infty)$ be a edge linear unimodal function on $G$. If $v_{max}$ is a vertex at which $f$ is maximized and the tree is rooted at $v_{max}$, then $f$ is non-increasing away from the root to its leaves.
\end{thm}

Computing $\ucat^1(f)$ can then be done in polynomial time using the following greedy algorithm. 

\begin{enumerate}
    \item Find a vertex $v$ of $G$ where $f$ attains its maximum value;
    \item Orient the edges of $G$ away from $v$ so that $G$ is a directed tree rooted at $v$;
    \item Define a function $h_{f,v} : G \rightarrow [0, \infty)$ such that $h_{f,v} = f(v)$;
    \item Define $h_{f,v}(w)$ for all other vertices $w$ inductively using the following rule. For each oriented edge $u \longrightarrow w$, find $h_{f,v}(u)$ and then set
    \[
        h_{f,v}(w) = \begin{cases}
            h_{f,v}(u),& \text{ if $f(w) > f(u)$} \\
            \max(h_{f,v}(u)-(f(u)-f(w)), 0), & \text{ otherwise }
        \end{cases};
    \]
    \item Compute the remainder function $R_v f = f - h_{f,v}$. Its support is a subforest in $G$;
    \item Repeat the above steps with components of $R_v f$ until $R_v f$ becomes the zero function.
\end{enumerate}

The list of functions $h_{f,v}$ generated by repeating these steps gives a minimal unimodal decomposition of $f$ in $O(\ucat^1(f) V) \leq O(V^2)$ time. 

\begin{thm}[Corollary 6.1 \cite{Bar20}]
Let $G=(V,E)$ be a tree and $f : G \rightarrow [0,\infty)$ be an edge linear function. A minimal unimodal decomposition of $f$ can be computed in $O(|V|^2)$ time.
\end{thm}

In addition, $\ucat_s^1(f) = \ucat^1(f)$ on trees, and the above algorithm is guaranteed to produce strongly unimodal components. Since $\ucat^p(f) = \ucat^1(f^p)$ for any $1 < p ,\infty$, they can also be computed with the same time complexity on trees. Determining $\ucat^{\infty}(f)$ on trees is a trivial problem; it is simply the number of local maxima of the function $f$.  

\section{Minimal Unimodal Decomposition on General Graphs}

The greedy algorithm for determining minimal unimodal decompositions on trees uses the fact that the global maxima of any function $f$ is guaranteed to be a mode in some minimal unimodal decomposition of it. This is not necessarily true in the case of a graph with cycles.

However, we can characterize unimodal functions on general graphs using the following obvious theorem.
\begin{thm}
\label{thm:unimod-char}
Let $f : G \rightarrow [0,\infty)$ be a edge linear function on $G$. $f$ is unimodal if and only if the subgraph induced by $\supp(f) = \{ v \in V \ | \ f(v) > 0 \}$ is a tree and the restriction of $f$ to $G(\supp(f))$ satisfies the conditions given in Theorem \ref{char-unimod-tree}.
\end{thm}

We now state the main result of this section. 

\begin{thm}
\label{th:main}
Let $G = (V,E)$ be a graph and $f : G \rightarrow [0,\infty)$ be an edge-linear function on $G$. For any $p \in \N \cup \{\infty\}$, the problem of determining whether $\ucat^p(f) \leq k$ is
\begin{itemize}
	\item solvable in $O(|V|)$ time if $k = 1$,
	\item NP-hard if $k \geq 2$.
\end{itemize} 
\end{thm}
The proof is split into three parts, addressing the cases $k=1$, $k \geq 3$ and $k = 2$.

\begin{proof}[Proof of Theorem \ref{th:main} when $k = 1$]
When $k = 1$, the problem reduces to determining whether $f$ is unimodal or not. By Theorem \ref{thm:unimod-char}, we first need to check if $\supp(f)$ is a tree. This can be done in $O(|V|)$ time, since the $\supp(f)$ can be found in $O(|V|)$. Then if the number of edges between them not equal to $|\supp(f)|-1$, we can return no immediately. Otherwise we find the maximum vertex in $\supp(f)$, and check that the function is non-increasing away from the maximum vertex in $O(|V|)$ time using depth first search.
\end{proof}

\begin{proof}[Proof of Theorem \ref{th:main} when $k \geq 3$]
For $k \geq 3$, the problem can be reduced in polynomial time to the Graph-$k$-coloring which is known to be NP-hard \cite{Gar79}. Let us form a new graph $\tilde{G}=(V\cup W, \tilde{E})$ by adding $k$ vertices $W = \{w_1, w_2, \ldots, w_k\}$ to $V$, an edge between each each pair of vertices in $W \times V$, and edges between each distinct pair of vertices in $W \times W$. We now take $\tilde{\mathbbm{1}} : \tilde{G} \rightarrow [0,\infty)$ to be the constant function taking value $1$ at all vertices and edges of $G$. Observe that $\ucat^p(\tilde{\mathbbm{1}}) = \ucat^1(\tilde{\mathbbm{1}})$ for any $1 < p < \infty$, so we can restrict our attention to $p \in \{1, \infty\}$. We show that $G$ can be $k$-colored if and only if $\ucat^p(\tilde{\mathbbm{1}}) \leq k$. This is established in the following Lemmas \ref{l:p1} and \ref{l:p2} which completes the proof for $k \geq 3$.
\end{proof}

\begin{lem}
\label{l:p1}
If $G$ can be $k$-colored, then $\ucat^1(\tilde{\mathbbm{1}}) \leq k$. 
\end{lem}
\begin{proof}
Let us suppose $G$ can be $k$-colored and let $V_i$ denote the subset of $V$ with color $i$ for $1 \leq i \leq k$. If we define $\tilde{V}_i := V_i \cup \{w_i\}$, the subgraph $\tilde{G}(\tilde{V}_i)$ is a tree for each $i$, which follows from these two facts:
\begin{itemize}
	\item $\tilde{G}(\tilde{V}_i)$ is connected since each vertex is connected to $w_i$,
	\item $\tilde{G}(\tilde{V}_i)$ cannot have cycles, since the vertices coming from $V_i$ cannot have any edges between each other, by definition of colorability. 
\end{itemize}
Let $\tilde{\mathbbm{1}}_i$ be the edge linear function taking value $1$ at vertices in $\tilde{V}_i$ and $0$ otherwise. By Theorem \ref{thm:unimod-char}, each $\tilde{\mathbbm{1}}_i$ is unimodal and $\sum_{i=1}^{k} \tilde{\mathbbm{1}}_i = \tilde{f}$, which means $\ucat^1(\tilde{\mathbbm{1}}) \leq k$. Similarly, let
$\tilde{\mathbbm{1}}_i$ be a function taking value $1$ at vertices in $\tilde{V}_i$ and $0$ elsewhere. Along an edge $(v_0, v_1)$, we define $\tilde{\mathbbm{1}}_i$ in the following two cases:
\begin{itemize}
\item If $\tilde{\mathbbm{1}}_i(v_0) = \tilde{\mathbbm{1}}_i(v_1)$, then $\tilde{\mathbbm{1}}_i(e) = \tilde{\mathbbm{1}}_i(v_0)$ for any point $e$ on the edge,
\item If $\tilde{\mathbbm{1}}_i(v_0) = 0$ and $\tilde{\mathbbm{1}}_i(v_1) = 1$, set $\tilde{\mathbbm{1}}_i(e) = 1$ for any point $e$ between $v_1$ and the midpoint of $v_0$ and $v_1$. Extend $\tilde{\mathbbm{1}}_i$ linearly between $v_0$ and the midpoint of $v_0$ and $v_1$.
\end{itemize}
Then each $\tilde{\mathbbm{1}}_i$ is unimodal and $\max_{i=1}^{k} \tilde{\mathbbm{1}}_i = \tilde{\mathbbm{1}}$ which means $\ucat^{\infty}(\tilde{\mathbbm{1}}) \leq k$. Therefore, if $G$ can be $k$-colored, the $\ucat^1(\tilde{\mathbbm{1}}) \leq k$ and $\ucat^{\infty}(\tilde{\mathbbm{1}}) \leq k$.
\end{proof}

\begin{lem}
\label{l:p2}
$\ucat^1(\tilde{\mathbbm{1}}) \leq k$, then $G$ can be $k$-colored.
\end{lem}
\begin{proof}
Assume $\ucat^1(\tilde{\mathbbm{1}}) \leq k$ or $\ucat^{\infty}(\tilde{\mathbbm{1}}) \leq k$. Let us take $\tilde{V}_i := \supp(\tilde{\mathbbm{1}}_i)$ for $1 \leq i \leq k$, where $\tilde{\mathbbm{1}}_i$ is the $i$-th unimodal component. Note that $\tilde{V}_i$ cover the vertices of $\tilde{G}$, but they need not be disjoint. Then $\tilde{G}(\tilde{V}_i)$ are trees by \ref{thm:unimod-char}. Each $\tilde{V}_i$ can contain at most two vertices from $W$, since three vertices from $W$ would lead to a cycle. The subsets $\tilde{V}_i$ can then be split into three groups:
\begin{itemize}[align=left, leftmargin=*]
\item[\textit{Subsets 1:}] $\tilde{V}_1, \ldots, \tilde{V}_{n_0}$ containing no vertices from $W$. This means that each $\tilde{V}_i$ contains only vertices from $V$ and the subgraph induced by these vertices in the original graph $G$ is a tree.
\item[\textit{Subsets 2:}] $\tilde{V}_{n_0+1}, \ldots, \tilde{V}_{n_0+n_1}$ containing $1$ vertex from $W$. Let $v_0, v_1$ be two distinct vertices from the original graph $G$ that belong to one such $\tilde{V}_i$. These vertices should not contain an edge between them in $G$, as that would lead to a cycle in $\tilde{G}(\tilde{V}_i)$ as both $v_0, v_1$ are also connected to the $1$ vertex from $W$.
\item[\textit{Subsets 3:}]  $\tilde{V}_{n_0+n_1+1}, \ldots, \tilde{V}_{n_0+n_1+n_2}$ containing $2$ vertices from $W$. These subsets cannot contain any other vertices, since any other vertex from $V$ is connected to both the vertices from $W$ leading to a cycle.
\end{itemize}
The $n_1$ subsets of type 2 containing exactly one vertex from $W$ each can cover at most $n_1$ vertices from $W$. The remaining vertices from $W$, which are at least $k-n_1$ in number, must be covered by the $n_2$ subsets containing exactly two vertices from $W$. This means $n_2 \geq \frac{k-n_1}{2}$ and since $n_0+n_1+n_2 = k$, $n_0 \leq  \frac{k-n_1}{2}$. We can now define $V_i = V \cap \tilde{V}_i$. The last $n_2$ such subsets of type 3 has to be empty which means $\{ V_i \}_{i=1}^{n_0+n_1}$ cover $V$, as $\tilde{V}_i$ cover $V \cup W$. However, $V_i$ need not be disjoint, and we make them disjoint by making the following modification. For any vertex $v$ in $V$, remove $v$ from all but one subset $V_i$. If we do this, we end up with two kinds of subsets $V_i$:
\begin{enumerate}
\item $V_1, \ldots, V_{n_0}$ such that $G(V_i)$ is a forest for each $i$. This is true since we potentially removed some vertices from a tree, which can only lead to a forest.
\item $V_{n_0+1}, \ldots, V_{n_0+n_1}$ such that no two vertices in a subset $V_i$ are adjacent.
\end{enumerate}
$n_1$ distinct colors can be assigned to the latter $n_1$ subsets. Any forest can be colored with two colors, and so we can assign two colors each to the former $n_0$ forests. Since $n_0 \leq \frac{k-n_1}{2}$, we need at most $k-n_1$ colors to do this, which means $G$ is $k$-colorable. 
\end{proof}

The proof for the case $k=2$ follows from the two propositions below.

\begin{prp}
\label{prp:k2p}
Let $G = (V,E)$ be a graph and $
\mathbbm{1}: G \rightarrow [0,\infty)$ be the constant function taking value $1$ at all points in $G$. Then 
\begin{itemize}
\item $\ucat^1(\mathbbm{1}) \leq 2$ if and only if $V$ can be split into two disjoint subsets $V_1, V_2$ such that $V_1 \cup V_2 = V$ and the subgraphs $G(V_1)$ and $G(V_2)$ are trees,
\item $\ucat^{\infty}(\mathbbm{1}) \leq 2$ if and only if $V$ can be split into two subsets $V_1, V_2$ such that $V_1 \cup V_2 = V$ and the subgraphs $G(V_1)$ and $G(V_2)$ are trees.
\end{itemize}
\end{prp}
\begin{proof}
Consider that case of $\ucat^1(\mathbbm{1})$. If $V$ can be split into two disjoint subsets $V_1, V_2$ such that $V_1 \cup V_2 = V$ and the subgraphs $G(V_1)$ and $G(V_2)$ are trees, we can define $\mathbbm{1}_i$ to be $1$ on the vertices $V_i$, $0$ otherwise, and extend to $G$ via edge-linearity. The functions are unimodal since $G(V_i)$ are trees and $\mathbbm{1}_1+\mathbbm{1}_2 = \mathbbm{1}$ on each vertex. Hence, $\ucat^1(\mathbbm{1}) \leq 2$.

If $\ucat^1(\mathbbm{1}) \leq 2$, let $\mathbbm{1}_1, \mathbbm{1}_2$ be unimodal functions summing to $\mathbbm{1}$. Let $V_i$ be $\supp(\mathbbm{1}_i)$, which means $G(V_i)$ are trees by Theorem \ref{thm:unimod-char}. If the supports of $\mathbbm{1}_1$ and $\mathbbm{1}_2$ are disjoint, we are done. Otherwise, modify $\mathbbm{1}_1, \mathbbm{1}_2$ at all vertices $v \in V_1\cap V_2$ such that $\tilde{\mathbbm{1}}_1(v) = 1$ and $\tilde{\mathbbm{1}}_2(v) = 0$ on $V_1\cap V_2$ . The resulting $\tilde{\mathbbm{1}}_1$ will be unimodal since $\tilde{\mathbbm{1}}_1(v) = 1$ on all vertices in its support, and $\supp(\tilde{\mathbbm{1}}_1) = \supp(\mathbbm{1}_1)$ which means $G(\supp(\tilde{\mathbbm{1}}_1))$ is a tree. $\tilde{\mathbbm{1}}_2$ will also be unimodal since $\tilde{\mathbbm{1}}_2(v) = 1$ on all vertices in its support, and $\supp(\tilde{\mathbbm{1}}_2) = \{ v \ | \ \mathbbm{1}_2(v) > 1-\epsilon \}$ for small enough $\epsilon$, which means $G(\supp(\tilde{\mathbbm{1}}_2)$ is also a tree, as $\mathbbm{1}_2$ is unimodal. Setting $V_i$ to $\supp(\tilde{\mathbbm{1}}_i)$ proves the proposition.

Now consider $\ucat^{\infty}(\mathbbm{1})$. If $V$ can be split into two subsets $V_1, V_2$ such that $V_1 \cup V_2 = V$ and the subgraphs $G(V_1)$ and $G(V_2)$ are trees, we can define $\mathbbm{1}_i$ to be $1$ on the vertices $V_i$, $0$ otherwise. We extend $\mathbbm{1}_i$ along an edge $(v_0, v_1)$ in the following two cases as in the earlier proof:
\begin{itemize}
\item If $\tilde{\mathbbm{1}}_i(v_0) = \tilde{\mathbbm{1}}_i(v_1)$, then $\tilde{\mathbbm{1}}_i(e) = \tilde{\mathbbm{1}}_i(v_0)$ for any point $e$ on the edge,
\item If $\tilde{\mathbbm{1}}_i(v_0) = 0$ and $\tilde{\mathbbm{1}}_i(v_1) = 1$, set $\tilde{\mathbbm{1}}_i(e) = 1$ for any point $e$ between $v_1$ and the midpoint of $v_0$ and $v_1$. Extend $\tilde{\mathbbm{1}}_i$ linearly between $v_0$ and the midpoint of $v_0$ and $v_1$.
\end{itemize}
 The functions are unimodal since $G(V_i)$ are trees and $\max(\tilde{\mathbbm{1}_1}, \tilde{\mathbbm{1}_2}) = 1$ on each vertex. Hence, $\ucat^{\infty}(\mathbbm{1}) \leq 2$.
 
If $\ucat^{\infty}(\mathbbm{1}) \leq 2$, let $\mathbbm{1}_1$ and $\mathbbm{1}_2$ be the unimodal components and let $V_i = \supp(\mathbbm{1}_i)$. By Theorem \ref{thm:unimod-char}, $G(V_i)$ are trees and they clearly cover $V$, which proves the result.
\end{proof}

\begin{prp}
\label{prp:dec}
If $G=(V,E)$ is a graph, the following two decision problems are NP-hard:
\begin{itemize}
\item Can $V$ be split into two disjoint subsets $V_1, V_2$ such that $V_1 \cup V_2 = V$ and the subgraphs $G(V_1)$ and $G(V_2)$ are trees ?
\item Can $V$ be split into two subsets $V_1, V_2$ such that $V_1 \cup V_2 = V$ and the subgraphs $G(V_1)$ and $G(V_2)$ are trees ?
\end{itemize}
\end{prp}
\begin{proof}
Let $G$ be a planar graph. The vertices $V$ can be split into two disjoint subsets such that each induced subgraph is a tree if and only if the dual graph is Hamiltonian. It is NP-hard to determine whether a planar graph is Hamiltonian \cite{Gar76}, which proves the first statement. 

The vertices $V$ can be split into two possibly overlapping subsets such that each induced subgraph is a tree if and only if the dual graph has a Hamiltonian path. It is NP-hard to determine if a planar graph has a Hamiltonian path \cite{Gar74} which proves the second
statement.
\end{proof}

\begin{proof}[Proof of Theorem \ref{th:main} when $k = 2$]
The result follows as a consequence of Propositions \ref{prp:k2p} and \ref{prp:dec}.
\end{proof}

The value of $\ucat^{\infty}(\mathbbm{1})$ is equal to the minimum number of open contractible sets that form a cover of the graph, also known as the geometric category of $G$ \cite{james_category_1978}. The supports of any unimodal decomposition of $\mathbbm{1}$ form an open contractible cover of $G$. On the other hand, given an open contractible cover of $G$, one can use bump functions taking value $1$ on each set in the cover to form a unimodal decomposition of the constant function.

\begin{cor}
\label{cor:cat}
The problem of determining whether a graph $G$ can be covered by $k$ open contractible sets is NP-hard for $k \geq 2$. 
\end{cor}

\section{Strongly Unimodal Decomposition on General Graphs}
 
 In this section, we consider the problem of finding the minimal strongly unimodal decomposition of a general graph. 

 
 \begin{thm}
 \label{th:main2}
 Let $G = (V,E)$ be a graph and $f : G \rightarrow [0,\infty)$ be a function on $G$. For any $p \in \mathbb{N}$, the problem of determining $\ucat_s^p(f)$ is NP-hard.
 \end{thm}

\begin{figure}
    \centering
    \includegraphics[width=0.9\linewidth]{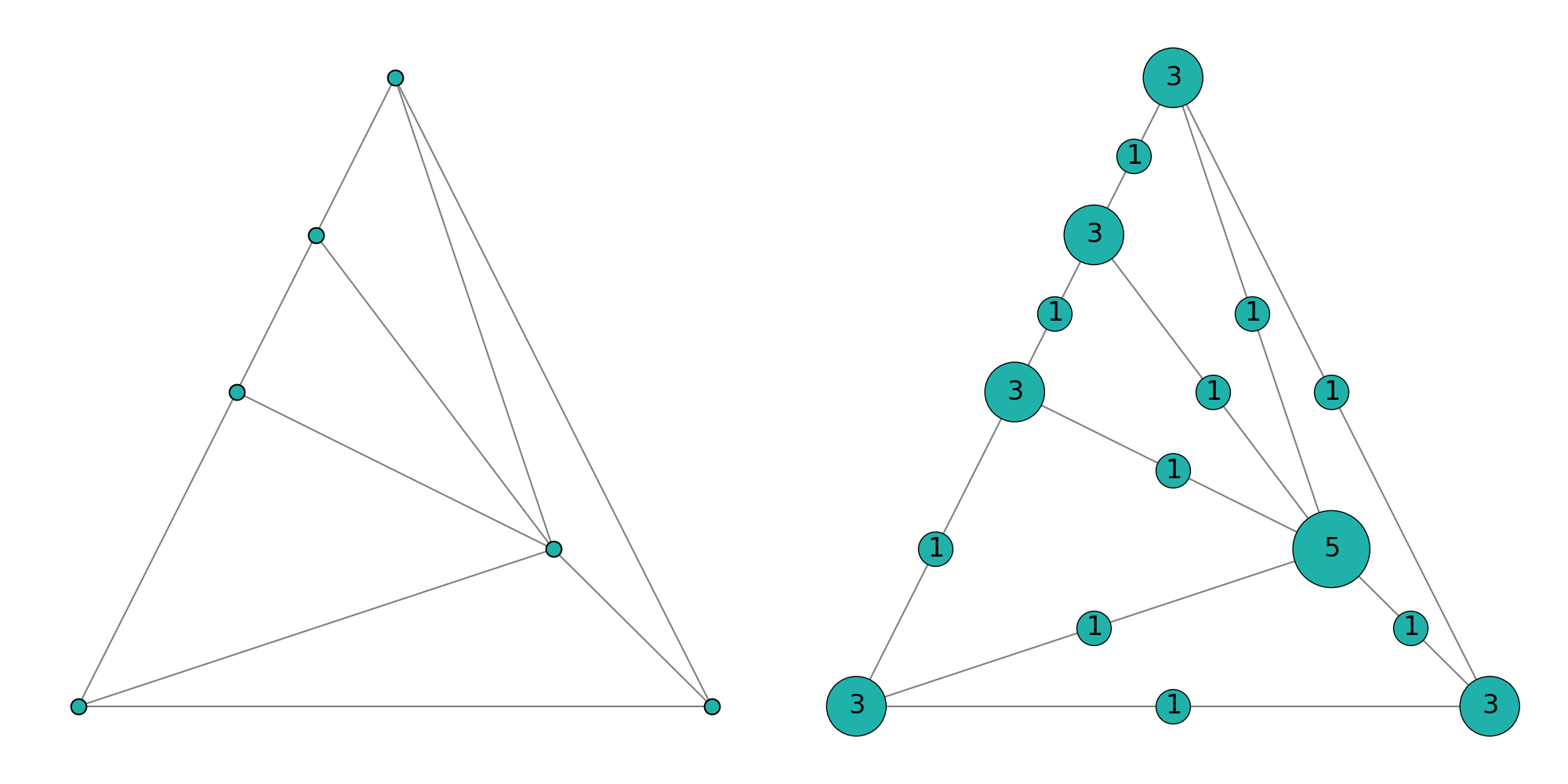}
    \caption{A graph $G$ (left) and the function $\tilde{f}$ on the augmented graph $\tilde{G}$, as defined in the proof of Theorem \ref{th:main2} (right).}
    \label{fig:main2_fig1}
\end{figure}
 
 \begin{proof}
 We prove by reducing the problem to the vertex cover problem which is known to be NP-complete. We will assume that the graph $G$ has girth greater than $4$. The vertex cover problem is NP-hard even with these restriction \cite{Mur92}. Given a graph $G$, we construct an auxiliary graph $\tilde{G}$ with an additional vertex at the midpoint of each edge of $G$ (see Figure \ref{fig:main2_fig1}). We then
 define the function $\tilde{f} : \tilde{G} \rightarrow [0,\infty)$ (see Figure \ref{fig:main2_fig1}) as follows: 
 \[
 \tilde{f}(v) = \begin{cases}
 \deg(v) & \ \text{ if $v \in G$}, \\
 1 & \ \text{ otherwise }
 \end{cases}
 \]
 We show that $G$ has a vertex cover with $k$ vertices if and only if $\tilde{f}$ has a strong unimodal decomposition consisting of $k$ components if and only if $\tilde{f}$ has a unimodal decomposition consisting of $k$ components in the following series of lemmas. This completes the reduction to vertex cover. Since the vertex cover problem is NP-hard on graphs with girth greater than $4$ \cite{Mur92}, determining minimal strong unimodal decomposition is also NP-hard.
\end{proof}

\begin{lem}
    If $G$ has a vertex cover with $k$ vertices, then $\tilde{f}$ has a unimodal decomposition consisting of $k$ components. If in addition, $G$ has girth greater than $4$, then $\tilde{f}$ has a strong unimodal decomposition consisting of $k$ components.
\end{lem}
\begin{proof}
 Suppose $G$ has a vertex cover with $k$ vertices. For each vertex in the cover $v_i$, define a function $f_i$ on $\tilde{G}$ as follows:
 \[
 	f_i(v) = \begin{cases}
	\deg(v) & \ \text{ if $v = v_i$ }, \\
	1 & \ \text{ if $v$ is a neighbor of $v_i$ in $G$ and $v$ is not in the vertex cover },\\
	1 & \ \text{ if $v$ is the midpoint of edge $e = (v_i, v_j)$ such that $v_j$ is not in the vertex cover },\\
	\frac{1}{2} & \ \text{ if $v$ is the midpoint of edge $e = (v_i, v_j)$ such that $v_j$ is in the vertex cover }, \\
	0 & \ \text{otherwise}
	\end{cases}
 \]
 
 We show that $\sum_i f_i(v) = f(v)$ for all $v \in \tilde{G}$. For a vertex $v_i$ in the vertex cover, only $f_i(v_i)$ is non-zero and $f_i(v_i) = f(v_i)$. For a vertex $v$ not in the vertex cover, the other endpoint of all edges incident on $v$ must be a vertex in the cover, by definition of a vertex cover. The function $f_i$ corresponding to each of these vertices takes value $1$ at $v$, and their sum equals $f(v) = \deg(v)$. The remaining vertices in $\tilde{G}$ represent midpoints of edges. Let $e = (v_i,v_j)$ be such an edge. At least one endpoint of this edge (say $v_i$) must be in the vertex cover, in which case $f_i(e) = 1$, and if both $v_i$ and $v_j$ are in the cover then $f_i(e)+f_j(e) = 1$. 
 
We now show that each $f_i$ is unimodal. The support of $f_i$ is contained in the set of $v_i$, its neighbors in $G$, and the midpoint vertices of the edges between $v_i$ and neighbors. The subgraph in $\tilde{G}$ induces by these vertices is homeomorphic to the closed star of $v_i$ in $G$, which is contractible to $v_i$. Hence, the support of $f_i$ is a tree. In addition, since $\deg(v_i) \geq 1$, $f_i$ decreases away from its mode vertex $v_i$ along any edge away from it, showing that each $f_i$ is unimodal. This proves the first part of the lemma.
 
We now show that the collection of functions $\{ f_i\}_{i=1}^{k}$ form a strongly unimodal collection if the girth of $G$ is greater than $4$. Since each superlevel set of $f_i$ has to be a tree, the intersection of any superlevel sets of $f_i$ must be a forest (disjoint union of trees) and we just need to show that they have to be empty or connected. Let the intersection $\cap_{l=1}^{m} \{f_{i_l} \geq a_{i_l} \}$ have two disconnected components. Let $w_1$ and $w_2$ be vertices of $\tilde{G}$ from two distinct components. Since $w_1$ and $w_2$ have a unique path between them in each superlevel set $\{f_{i_l} \geq a_{i_l} \}$, we can assume that there are two indices $i_{l_1}$ and $i_{l_2}$ such that $w_1$ and $w_2$ belong to two different components of $\{f_{i_{l_1}} \geq a_{i_{l_1}} \} \cap \{f_{i_{l_2}} \geq a_{i_{l_2}} \}$. There are three possible cases:
\begin{itemize}
    \item $w_1$ and $w_2$ are midpoints of edges. If this is the case, both $w_1$ and $w_2$ should have endpoints at $v_{i_{l_1}}$ and $v_{i_{l_2}}$ which isn't possible, as we don't allow multiple edges in $G$.
    \item $w_1$ is a vertex of $G$ and $w_2$ is a midpoint of an edge. This means that $v_{i_{l_1}}$ and $v_{i_{l_2}}$ are endpoints of the edge $w_2$, and $w_1$ is a common neighbor of both these vertices in $G$. This implies the existence of a cycle of length $3$ in $G$.
    \item $w_1$ and $w_2$ are vertices of $G$. In this case, both $w_1$ and $w_2$ are neighbors of both $v_{i_{l_1}}$ and $v_{i_{l_2}}$ in $G$. This implies the existence of a cycle of length $4$ in $G$.
\end{itemize}
Hence, if $G$ has girth greater than $4$, then $\{f_i\}_{i=1}^{k}$ form a strongly unimodal decomposition.
\end{proof}

\begin{lem}
\label{lem:umd}
    If $\tilde{f}$ has a unimodal decomposition with $k$ components, then $G$ has a vertex cover with at most $k$ components.
\end{lem}
\begin{proof}
Now assume that $\tilde{f}$ has a unimodal decomposition with $k$ components $F = \{f_1, \ldots, f_k\}$. We pick at most one mode from each component to form a set of vertices in $G$. If the only mode of a component function is a vertex in $\tilde{G} - G$, we ignore this mode. If a component has multiple modes, with at least one mode in $G$, we pick any one of those modes in $G$. We will prove that the set of vertices in $G$ we obtain in this manner forms a vertex cover with at most $k$ elements. We will prove by contradiction. Let us assume that there are two vertices in $G$, $v_1$ and $v_2$ such that both the vertices are modes of none of the components. We will show that these two vertices can't be adjacent in $G$. Let $v$ be a vertex that is not the mode of any function in $F$, and $E_v = \{e_1, \ldots, e_{\deg(v)}\}$ be the edges hitting $v$ in $G$. Define $F^{v}_{e}$ to be the set of functions
\[
	F^{v}_{e} = \{ f_i \in F \ | f_i(e)  \geq f_i(v) \geq f_i(e^{'}) \ \forall e^{'} \neq e \}.
\]
The union of the collection $\{ F^{v}_{e} \}_{e \in E_v}$ should equal $F$, since $v$ is not the mode of any function in $F$. In addition, 
\[
	\sum_{f \in F^{v}_{e}} f(v) \leq \sum_{f \in F^{v}_{e}} f(e) \leq 1.
\]
which means 
\begin{equation}
\label{impsum1}
	\deg(v) = \sum_{f \in F} f(v) = \sum_{e \in E_v} \left( \sum_{f \in F^{v}_{e}} f(v) \right) \leq \sum_{e \in E_v} 1 = \deg(v).
\end{equation}
This implies that for each $e \in E_v$,
\[
	\sum_{f \in F^{v}_{e}} f(v) = 1.
\]
In addition, if $e_1 \neq e_2 \in E_v$ and $f \in F^{v}_{e_1}$ and $f(e_1) > 0$, then $f(e_2) = 0$. If $f(e_2) > 0$, there are two possibilities:
\begin{itemize}
	\item Either $f(e_2) <  f(e_1)$. In this case, $f \notin F^{v}_{e_2}$ but $f(e_2) = a > 0$, which means
	\[
		\sum_{g \in F^{v}_{e_2}} g(e_2) = 1 - \sum_{g \in F-F^{v}_{e_2}} g(e_2) \leq 1-f(e_2) < 1, \quad \text{since $f \notin F^{v}_{e_2}$,}
	\] 
	which further implies
	\[
		\deg(v) = \sum_{g \in F} g(v) \leq \sum_{e \in E_v} \left( \sum_{g \in F^{v}_{e}} g(v) \right)  < \deg(v)
	\]
	giving a contradiction.
	\item Or $f(e_2) = f(v) = f(e_1) > 0$. In this case $f \in F^v_{e_2} \cap F^v_{e_1}$, which means
	\[
		\deg(v) = \sum_{g \in F} g(v) \leq \sum_{e \in E_v} \left( \sum_{g \in F^{v}_{e}} g(v) \right)  - f(v) \leq \deg(v) - f(v) < \deg(v),
	\]
	where the first inequality follows from the fact that when $F$ is written as $\cup_{e \in E_v} F^{v}_{e}$, the function $f$ appears in both $F^v_{e_2}$ and $F^v_{e_1}$, which means it is counted twice and can be subtracted once from the sum on the right side of the inequality. This again gives a contradiction.
\end{itemize}
To summarize, the collection of functions $\{ F^{v}_{e} \}_{e \in E_v}$ satisfy the following two properties:
\begin{equation}
\label{prop1-mode-free}
\cup_{e \in E_v} F^{v}_{e} = F,
\end{equation}
\begin{equation}
\label{prop3-mode-free}
\sum_{f \in F^{v}_{e}} f(v) = 1,
\end{equation}
and
\begin{equation}
\label{prop2-mode-free}
f \in F^{v}_{e} \implies f(e^{'}) = 0 \ \text{ for all $e^{'} \in E_v - \{e\}$}.
\end{equation}

Now suppose $v_1$ and $v_2$ are vertices in $G$ with edge $e_{1,2}$ hitting both vertices. If both $v_1$ and $v_2$ are not the modes of any functions in $F$, then the set of functions $F^{v_1}_{e_{1,2}}$ = $F^{v_2}_{e_{1,2}} $. This is true as $f(e_{1,2}) \geq f(v_1)$ implies $f(e_{1,2}) \geq f(v_2)$. To see this, observe that if $f(v_2) > f(e_{1,2})$ then $f \in F^{v_2}_{e^{'}}$ for a different edge $e^{'}$ and by property \eqref{prop2-mode-free} $f(e_{1,2})$ would be zero. 
In addition, any function in the set $F^{v_1}_{e_{1,2}}$ must take value zero on any vertex in $\tilde{G}$ other than $v_1, e_{1,2}, v_2$, also due to property \eqref{prop2-mode-free} and unimodality of $f$. Now we can see that
\begin{equation}
\label{finalprop}
	\sum_{f \in F^{v_1}_{e_{1,2}}} f(v_1) = \sum_{f \in F^{v_1}_{e_{1,2}}} f(v_2) = \sum_{f \in F^{v_1}_{e_{1,2}}} f(e_{1,2}) = 1, \text{ \ and \ } \sum_{f \in F^{v_1}_{e_{1,2}}} f(v) = 0 \ \text{ for all other $v \in \tilde{G}$} 
\end{equation}
This means we can remove the set of functions $F^{v_1}_{e_{1,2}}$ from $F$ and replace it with a single function $\sum_{f \in F^{v_1}_{e_{1,2}}} f$, a  unimodal function as it takes the value $1$ on a path containing three vertices and value $0$ on all other vertices. This modification can be done in polynomial time, by scanning through each vertex in $\tilde{G}-G$ to check if it satisfies \eqref{finalprop}, which can take at most $O((E + V)^3)$ time. After this modification, both $v_1$ and $v_2$ are modes of the new function $\sum_{f \in F^{v_1}_{e_{1,2}}} f$, which is a contradiction.
\end{proof}

Notice that in the proof of Lemma \ref{lem:umd}, we only required $\tilde{f}$ to have $k$ unimodal components, not strongly unimodal components. If $\tilde{f}$ had a strongly unimodal decomposition with $k$ functions, it would automatically have a unimodal decomposition with $k$ functions as strong unimodal implies unimodal. We have essentially shown that the minimum vertex cover of $G$ is equal to $\ucat^1(\tilde{f})$ when $G$ is an arbitrary graph, and that $\ucat^1(\tilde{f}) = \ucat^1_s(\tilde{f})$ if $G$ has girth greater than $4$, proving Theorem \ref{th:main2}.

This also means that Theorem \ref{th:main2} is an alternate proof of the NP-hardness of determining $\ucat^p(f)$. However, Theorem \ref{th:main} shows the stronger result that even for any fixed $k \geq 2$, the decision problem $\ucat^p(f) \leq k$ is NP-hard. Since vertex cover is not NP-hard when $k$ is fixed, the same cannot be said about Theorem \ref{th:main2}. In fact, the tractability of determining $\ucat_s^p(f) \leq k$ for a fixed $k$ is quite different from that of $\ucat^p(f) \leq k$, at least for small values of $k$, as shown in the next theorem.

\begin{thm}
\label{th:diff-s}
For any $p \in \mathbb{N}\cup\{\infty\}$, the problem of determining whether $\ucat_s^p(f) \leq k$ is solvable in polynomial time if $k = 2$. 
\end{thm}
\begin{proof}
    $\ucat_s^p(f) > 2$ for any cyclic graph, as the supports of the strongly unimodal components form a good cover of the graph $G$. The number of elements in a good cover of a topological space with one cycle is greater than 2 \cite{Kar17}.
    Hence, checking $\ucat_s^p(f) \leq 2$ is equivalent to checking if $G$ is a tree and then checking if $\ucat_s^p(f) \leq 2$ on the tree, which can be done in polynomial time.
\end{proof}

Since vertex cover is NP-hard even on planar graphs with degree at most 3 \cite{Mur92}, and the modifications we made in the above proofs preserve planarity and maximum degree, we get the following corollary.

\begin{cor}
\label{cor:cor1}
Let $G = (V,E)$ be a planar graph with degree at most 3 and $f : G \rightarrow [0,\infty)$ be an edge linear function on $G$. The problems of determining $\ucat^p(f)$ and $\ucat_s^p(f)$ are both NP-hard.
\end{cor}

It is also known that vertex cover is NP-hard to solve approximately within a factor less than $\sqrt{2}$ \cite{Din18}, and this factor can be improved to $2$ if the UGC is true \cite{khot2008}, which gives the corollary 
\begin{cor}
\label{cor:cor2}
For any $p \in \mathbb{N}$, it is NP-hard to approximate $\ucat^p(f)$ or $\ucat_s^p(f)$ upto a factor less than $\sqrt{2}$, and $2$ if UGC is true.
\end{cor}

\section{Minimal Unimodal Decomposition in Higher Dimensions}

In this section, we assume that $G$ is a simplicial complex of dimensional greater than 1 and $f : G \rightarrow [0,\infty)$ a piecewise linear function. The problem of determining minimal unimodal decompositions in this situation is much harder. In dimensions $d \geq 4$ the problem of determining whether a function is unimodal or not, a problem that can be solved in linear time on graphs, is not just NP-hard but undecidable.

\begin{thm}
Let $G$ be a simplicial complex of dimension $d \geq 4$ and  $f : G \rightarrow [0,\infty)$ a simplicial function. Then it is undecidable to determine whether $f$ is unimodal or not.
\end{thm}
\begin{proof}
Consider the constant function $\mathbbm{1}$ on $G$. $\mathbbm{1}$ is unimodal if and only if $G$ is contractible. The problem of determining whether a simplicial complex of dimension greater than 3 is contractible is undecidable \cite{Sti12, Tan16}.
\end{proof}

The undecidability of determining contractibility of simplicial complexes of dimension $2$ and $3$ is an open problem \cite{Tan16}. Nevertheless, the minimal unimodal decomposition problem is still NP-hard in dimensions $2$ and $3$, even in the simplest case of a simplicial approximation of the unit square in $\R^2$.

\begin{thm}
\label{thm:highd}
Let $G$ is a simplicial complex of dimension $2$ that is homeomorphic to $[0,1]\times[0,1] \subset \R^2$ and $f : G \rightarrow [0,\infty)$ a simplicial function. The problems of determining $\ucat^p(f)$ or $\ucat_s^p(f)$ are both NP-hard for any $p \in \mathbb{N}$.
\end{thm}
\begin{proof}
Take a planar graph $\tilde{G} = (\tilde{V}, \tilde{E})$ and an edge linear function $\tilde{f}$ on $\tilde{G}$. We can embed $\tilde{G}$ in $[0,1]\times[0,1]$ and extend it to a simplicial approximation $G$ of $[0,1]\times[0,1]$ with vertices $V$ such that $\tilde{G}$ is a subcomplex of $G$. If $v_1, v_2, v_3$ are vertices in $\tilde{V}$ and $(v_1,v_2,v_3)$ is a 2-face in $G$, then we add a vertex $v$ to $V$ in the center of this face and add edges $(v_i, v)$. This means the face $(v_1,v_2,v_3)$ is replaced with three faces $(v,v_2,v_3), (v,v_1,v_2), (v,v_1,v_3)$ and this ensures that no face in $G$ has all three vertices coming from $\tilde{G}$. We can now extend $\tilde{f}$ to a simplicial function $f$ on $G$ by setting the value of $f$ to zero for all vertices outside $\tilde{G}$ in $G$, in $V-\tilde{V}$. We show that $\tilde{f}$ and $f$ have the same number of minimal unimodal components, from which the theorem follows.

$\{f \geq c\}$ is contractible if and only if the subcomplex $G(V_c)$ induced by the vertices $V_c = \{ v \in V \ | \ f(v) \geq c\}$ in $G$ is contractible, since $\{f \geq c\}$ is homotopy equivalent to the open star of $V_c$, 
which in turn is homotopic to $G(V_c)$. Since $f$ is zero on vertices outside $\tilde{V}$, we know that $V_c \subset \tilde{V}$. Also, no three vertices from $\tilde{V}$ form a 2-face in $G$, and so $G(V_c) \subset G(\tilde{V}) = \tilde{G} $ 
has to be a graph and $G(V_c) = \tilde{G}(V_c)$. This means that $\{f \geq c\}$ is contractible if and only if $\{ \tilde{f} \geq c \}$ is contractible, which means that $f$ is unimodal if and only if $\tilde{f}$ is unimodal. In addition, if we have finite intersections $\{ f_i \geq c_i \}_{i=1}^{k}$, a similar argument shows that any component of this intersection can be homotoped to the intersection $\{ \tilde{f_i} \geq c_i \}_{i=1}^{k}$ of the restriction to $\tilde{G}$.
This means that the restriction of any (strongly) unimodal 
decomposition of $f$ to $\tilde{G}$ is going to be a (strongly) unimodal decomposition of $\tilde{f}$, and any (strongly) unimodal decomposition of $\tilde{f}$ can be extended to a (strongly) unimodal decomposition of $f$ by setting the value of components to zero on vertices outside $\tilde{V}$, proving the theorem.

\end{proof}

\section{Conclusion and Future Work}
In this article, we have shown that the minimal unimodal decomposition problem, along with several closely related variants, is NP-hard on cyclic graphs and two(or more)-dimensional simplicial complexes. Several promising directions remain for future research:

\begin{enumerate}
\item \textit{Approximation Algorithms}: By reducing the problem to Vertex Cover, we established that computing $\ucat^p(f)$ within a factor better than 2 is NP-hard. Nevertheless, Vertex Cover admits a well known 2-approximation via a greedy algorithm. The greedy method proposed in \cite{Bar20} for computing $\ucat^p(f)$ on trees can be adapted to cyclic graphs, yielding a unimodal decomposition that is not necessarily minimal. Since the motivation for minimal decompositions stems from applications in topological statistics \cite{baryshnikov_unimodal_2011}, understanding the approximation quality of such algorithms would help assess their practical utility in statistical problems.

\item \textit{Fixed-Parameter Tractability}: While $\ucat^p(f)$ can be computed in polynomial time on trees, it remains open whether efficient algorithms exist for broader classes of graphs. In particular, it would be interesting to investigate whether $\ucat^p(f)$ can be computed in polynomial time on graphs that are nearly trees, such as graphs of bounded treewidth, potentially via dynamic programming techniques as in \cite{bodlaender88}. From a topological standpoint, it would also be interesting to explore whether bounded first Betti number ($\beta_1$) enables efficient computation. An $O(|V|^3)=O(|V|^{2+\beta_1}$ algorithm is already available in literature for graphs homeomorphic to the circle $S^1$ (see \cite{Gov21}). It would be interesting to see if the $O(|V|^{2+\beta_1})$ bound that holds for trees and cyclic graphs can be generalized to all graphs.

\end{enumerate}

\bibliographystyle{alpha}
\bibliography{references}
\end{document}